\documentclass[leqno]{amsart}
\usepackage{amsmath}
\usepackage{amssymb}
\usepackage{amsthm}
\usepackage{graphicx}
\usepackage{enumerate}
\usepackage[mathscr]{eucal}
\theoremstyle{plain}
\newtheorem{theorem}{Theorem}[section]

\newtheorem{cor}{Corollary}[theorem]
\newtheorem{lemma}{Lemma}[section]
\newtheorem{conj}[theorem]{Conjecture}

\theoremstyle{definition}
\newtheorem{definition}{Definition}[section]
\newtheorem{remark}{Remark}[section]

\newtheorem{example}{Example}[theorem]

\usepackage[pagewise]{lineno}

\begin{document}

\title[Some remarks on orthogonality of operators]{Some remarks on orthogonality of  bounded linear operators}
\author[Ray, Sain, Dey and Paul ]{Anubhab Ray, Debmalya Sain, Subhrajit Dey and Kallol Paul}

\newcommand{\acr}{\newline\indent}

\address[Ray]{Department of Mathematics\\ Jadavpur University\\ Kolkata 700032\\ West Bengal\\ INDIA}
\email{anubhab.jumath@gmail.com}

\address[Sain]{Department of Mathematics\\ Indian Institute of Science\\ Bengaluru 560012\\ Karnataka \\India\\ }
\email{saindebmalya@gmail.com}

\address[Dey]{Department of Mathematics\\ Muralidhar Girls' College \\ Kolkata 700029 \\ West Bengal\\ INDIA}
\email{subhrajitdeyjumath@gmail.com}

\address[Paul]{Department of Mathematics\\ Jadavpur University\\ Kolkata 700032\\ West Bengal\\ INDIA}
\email{kalloldada@gmail.com}

\thanks{Mr. Anubhab Ray would like to thank DST, Govt. of India, for the financial support in the form of doctoral fellowship. Dr. Debmalya Sain would like to thank UGC, Govt. of India for D. S. Kothari Post-doctoral Fellowship. Dr. Sain is indebted to Professor Gadadhar Misra for his motivating inspirations. The research of Prof. Kallol Paul is partially supported by RUSA 2.0, JU.}


\begin{abstract}
We explore the relation between the orthogonality of bounded linear operators in the space of  operators and that of elements in the ground space.  To be precise, we study  if $ T, A   \in \mathbb{L}(\mathbb{X}, \mathbb{Y}) $  satisfy $ T \bot_B A,$ then whether there exists $ x \in \mathbb{X} $   such that $  Tx\bot_B Ax$   with $ \|x\| =1, \|Tx\| = \|T\|$, where $\mathbb{X}, \mathbb{Y} $ are normed linear spaces.  
  In this context, we introduce the notion of Property $ P_n $ for a Banach space and illustrate its connection with orthogonality of a bounded linear operator between Banach spaces. We further study Property $ P_n $ for various polyhedral Banach spaces.

\end{abstract}

\maketitle

\section{Introduction}

The purpose of the present article is to continue the study of orthogonality properties of bounded linear operators between Banach spaces, in light of the seminal result obtained by Bhatia and \v{S}emrl \cite{BS} regarding orthogonality of linear operators on Euclidean spaces. Let us first establish the relevant notations and the terminologies in this context.\\

Letters $\mathbb{X}$ and $\mathbb{Y}$ denote Banach spaces. Throughout the article, we work only with real Banach spaces. Let $B_{\mathbb{X}}=\{x\in \mathbb{X}: \|x\|\leq 1\}$ and $S_{\mathbb{X}}=\{x\in \mathbb{X}:\|x\|=1\}$ denote the unit ball and the unit sphere of $\mathbb{X}$ respectively. Let $ E_{\mathbb{X}} $ denotes the set of all extreme points  of $B_{\mathbb{X}}. $ For a set $ \mathcal{S} \subset \mathbb{X}, $ $|\mathcal{S}|$ denotes the cardinality of $ \mathcal{S}. $ Let $\mathbb{L}(\mathbb{X},\mathbb{Y})$ denote the Banach space of all bounded linear operators from $\mathbb{X}$ to $\mathbb{Y},$ endowed with the usual operator norm. We write $ \mathbb{L}(\mathbb{X},\mathbb{Y}) = \mathbb{L}(\mathbb{X}), $ if $ \mathbb{X}= \mathbb{Y}. $ For a bounded linear operator $T\in \mathbb{L}(\mathbb{X},\mathbb{Y}), $ let $ M_T $ denote the norm attainment set of $ T, $ i.e., $M_T=\{x\in S_\mathbb{X}:\|Tx\|=\|T\|\}.$ The notion of Birkhoff-James orthogonality in a Banach space is well-known and is used extensively in the study of the geometry of Banach spaces. For $ x,y \in \mathbb{X}, $ $ x $ is said to be orthogonal to $ y $ in the sense of Birkhoff-James \cite{J}, written as $ x \bot_{B} y, $ if $ \|x+ \lambda y\| \geq \|x\| $ for all $ \lambda \in \mathbb{R}. $ Similarly, for $ T, A \in \mathbb{L}(\mathbb{X},\mathbb{Y}), $ $ T $ is said to be Birkhoff-James orthogonal to $ A, $ written as $ T \bot_{B} A, $ if $ \|T+ \lambda A\| \geq \|T\| $ for all $ \lambda \in \mathbb{R}. $ For the $ n$-dimensional Euclidean space $ \mathbb{E}^{n}, $ Bhatia and \v{S}emrl \cite{BS} proved that for $ T, A \in \mathbb{L}(\mathbb{E}^{n}), $ $ T \bot_{B} A $ if and only if there exists $ x \in S_{\mathbb{E}^{n}} $ such that $ \|Tx\|= \|T\| $ and $ Tx \bot_{B} Ax. $ We refer the readers to \cite{P,PHD} for another approach in this context. In recent times, various generalizations of this remarkable theorem has been obtained \cite{Sa,SP} in the setting of Banach spaces. We note that the sufficient part of the above theorem is true whenever the domain space and the co-domain space are normed linear spaces of any dimensions, and not just Euclidean spaces. On the other hand, the necessary part of the said theorem is not true in general Banach spaces, even  in the finite-dimensional case \cite{LS, SP, SPH}. In view of this, the authors \cite{SPH} defined the Bhatia-\v{S}emrl (B\v{S}) Property of a bounded linear operator $ T \in \mathbb{L}(\mathbb{X}). $ It is possible to extend the definition of the B\v{S} Property in a broader context, by not imposing the restriction that the domain space and the co-domain space must be identical. Therefore, let us first make note of the following definition of the B\v{S} Property in its full generality.

\begin{definition}
Let $\mathbb{X}, \mathbb{Y} $ be Banach spaces and let $ T \in \mathbb{L}(\mathbb{X},\mathbb{Y}). $ We say that $ T $ satisfies the Bhatia-\v{S}emrl (B\v{S}) Property if for any $ A \in \mathbb{L}(\mathbb{X},\mathbb{Y}), $ $ T \bot_{B} A $ implies that there exists $ x \in M_T $ such that $ Tx \bot_{B} Ax. $
\end{definition}

Our aim is to study the B\v{S} Property of bounded linear operators between Banach spaces, especially when the domain space and the co-domain space are polyhedral. $ \mathbb{X} $ is said to be a polyhedral Banach space if $ B_{\mathbb{X}} $ has only finitely many extreme points. Equivalently, $ \mathbb{X} $ is a polyhedral Banach space if $ B_{\mathbb{X}} $ is a polyhedron. In this context, let us mention the following formal definitions:

\begin{definition}
A polyhedron $P$ is a non-empty compact subset of $\mathbb{X}$ which is the intersection of finitely many closed half-spaces of $\mathbb{X}$, i.e., $P=\bigcap_{i=1}^r M_i,$ where $ M_i $ are closed half-spaces in $ \mathbb{X} $ and $r \in \mathbb{N}.$ The dimension of the polyhedron $P$, written as $ \dim P, $ is defined as the dimension of the subspace generated by the differences $ v-w $ of vectors $ v,w \in P. $
\end{definition}

\begin{definition}
A polyhedron $ Q $ is said to be a face of the polyhedron $ P $ if either $ Q=P $ or if we can write $ Q = P \cap \delta M, $ where $ M $ is a closed half-space in $ \mathbb{X} $ containing $ P $ and $ \delta M $ denotes the boundary of $  M. $ If $ \dim Q = i, $ then $  Q $ is called an $i$-face of $ P. $ $ (n-1)$-faces of $ P $ are called facets of $ P $ and $1$-faces of $ P $ are called edges of $ P. $
\end{definition}

\begin{definition} \cite{SPBB}
Let $ \mathbb{X} $ be a finite-dimensional polyhedral Banach space. Let $ F $ be a facet of the unit ball $ B_{\mathbb{X}} $ of $  \mathbb{X}. $ A functional $ f \in S_{\mathbb{X}^{*}} $ is said to be a \emph{supporting functional corresponding to the facet $ F $ of the unit ball $ B_{\mathbb{X}} $} if the following two conditions are satisfied:\\
$ (1) $ $ f $ attains norm at some point $ v $ of $ F, $\\
$ (2) $ $ F =(v + \ker f)\cap S_{\mathbb{X}}. $
\end{definition}
We also make use of the concept of normal cones in a Banach space in our study.
\begin{definition}
A subset $ K $ of $ \mathbb{X} $ is said to be a normal cone in $ \mathbb{X} $ if \\
$ (i)~ K + K \subset K, (ii)~ \alpha K \subset K $ for all $ \alpha \geq 0, $ and $ (iii)~ K \cap (-K) = \{\theta\}. $
\end{definition}
Normal cones are important in the study of the geometry of Banach spaces, because there is a natural partial ordering $ \geq $ associated with a normal cone $ K. $ Namely, for any two elements $ x, y \in \mathbb{X},~ x \geq y $ if $ x-y \in K. $ It is easy to observe that in a two-dimensional Banach space $ \mathbb{X}, $ any normal cone $ K $ is completely determined by the intersection of $ K $ with the unit sphere $ S_{\mathbb{X}}. $  Keeping this in mind, when we say that $ K $ is a normal cone in $ \mathbb{X}, $ determined by $ v_1, v_2, $ what we really mean is that $ K \cap S_{\mathbb{X}} = \left\{\frac{(1-t)v_1 + tv_2}{\| (1-t)v_1 + tv_2 \|} : t \in[0,1]\right \}. $ Of course, in this case $ K = \{\alpha v_1 + \beta v_2 : \alpha, \beta \geq 0\}. $

A complete characterization of linear operators on a two-dimensional real Banach space satisfying the B\v{S} Property has been obtained in \cite{SPH} by proving the following theorem:
\begin{theorem}
A linear operator $ T $ on a two-dimensional real Banach space $ \mathbb{X} $ satisfies the B\v{S} Property if and only if the set of unit vectors on which $ T $ attains norm is connected in the corresponding projective space $ \mathbb{R}P^1 = S_{\mathbb{X}} / \{x \sim -x\}. $  
\end{theorem}
Equivalently, we can say in the two-dimensional case that $ T \in \mathbb{L}(\mathbb{X}) $ satisfies the B\v{S} Property if and only if $M_T = D \cup (-D), $ where $ D $ is a closed connected subset of $ S_{\mathbb{X}}. $ In this context, we refer the readers to the following conjecture \cite{SPH}, which is yet to be proved or disproved, when the dimension of the space is greater than two. 
\begin{conj}
	A linear operator $T$ on an $ n $-dimensional real  Banach space $\mathbb{X}$ satisfies the B\v{S} Property if and only if the set of unit vectors on which T attains norm is connected in the corresponding projective space $\mathbb{R}P^{n-1} = S_{\mathbb{X}}/{ \{x \sim -x \}}. $
\end{conj}
 However, the sufficient part of the conjecture was proved in \cite{SP}. Indeed, if $\mathbb{X}$ is a finite-dimensional real Banach space and $ T \in \mathbb{L}(\mathbb{X}) $ is such that $M_T = D \cup (-D)$, where $ D $ is a closed connected subset of $ S_{\mathbb{X}}, $  then $ T $ satisfies the B\v{S} Property. We are interested in exploring the converse to the above result. To be precise, if $ \dim{\mathbb{X}} > 2 $ and $M_T \neq D \cup (-D), $ where $ D $ is a closed connected subset of $ S_{\mathbb{X}}, $ then whether $ T \in \mathbb{L}(\mathbb{X}, \mathbb{Y}) $ satisfies the B\v{S} Property, where $ \mathbb{Y} $ is any Banach space. With this motivation in mind, we introduce the following definition for a Banach space $ \mathbb{X},$ which plays a significant role in our study.

\begin{definition}
Let $ \mathbb{X} $ be a Banach space. Given $ n \in \mathbb{N}, $ we say that $ \mathbb{X} $ has Property $ P_n $ if for every choice of $ n $ vectors $ x_1, x_2, \ldots, x_n \in S_{\mathbb{X}}, $ $ \bigcup \limits_{i=1}^{n} x_{i}^{\bot} \subsetneqq \mathbb{X}. $ 
\end{definition}

It is clear from the above definition that if $ \mathbb{X} $ has Property $ P_n $ then $ \mathbb{X} $ has Property $ P_m $ for all $ m \in \mathbb{N}, $ with $ m \leq n. $ We illustrate the connection between Property $ P_n $ for a Banach space and bounded linear operators not satisfying the B\v{S} Property. We further explore Property $ P_n $ for different polyhedral Banach spaces. 
 
\section{Main Results}

We begin this section with the observation that Theorem $ 2.3 $ of \cite{SPH} holds true even if the co-domain space is any Banach space with dimension at least two. Indeed, the said theorem can be stated in the following more general form, by using essentially the same arguments presented in the proof of the original result. 

\begin{theorem} \label{th:BS-property}
Let $ \mathbb{X} $ be a two-dimensional Banach space and let $ \mathbb{Y} $ be a Banach space of dimension greater than or equal to two. Let $ T \in \mathbb{L}(\mathbb{X},\mathbb{Y}) $ be such that $ M_T $ has more than two components. Then $ T $ does not satisfy the B\v{S} Property.
\end{theorem}

As a corollary to Theorem \ref{th:BS-property}, we can provide an elementary condition on $ M_T $ so that $ T $ does not satisfy the B\v{S} Property, when $ \mathbb{X} $ is a two-dimensional Banach space.
 
\begin{cor} \label{cor:BS-2-dim}
Let $ \mathbb{X} $ be a two-dimensional Banach space and let $ \mathbb{Y} $ be a Banach space of dimension greater than or equal to two. Let $ T \in \mathbb{L}(\mathbb{X},\mathbb{Y}) $ be such that there exist $ x, y \in M_T $ with $ x \neq \pm y $ and $ \frac{x+y}{\|x+y\|}, \frac{x-y}{\|x-y\|} \notin M_T. $ Then $ T $ does not satisfy the B\v{S} Property.
\end{cor}

\begin{proof}
We claim that $ M_T $ has more than two components. Let $ u_1 = \frac{x+y}{\|x+y\|} $ and $ u_2 = \frac{x-y}{\|x-y\|}. $ Consider the following subsets of $ S_{\mathbb{X}}: $
\[ S_1 = \left\{ \frac{(1-t)u_1+tu_2}{\|(1-t)u_1+tu_2\|} : t \in (0,1) \right\},\]
\[ S_2 = \left\{ \frac{(1-t)u_1+t(-u_2)}{\|(1-t)u_1-tu_2\|} : t \in (0,1) \right\},\]
\[ S_3 = -S_1 ~ \mbox{and} ~ S_4 = -S_2. \]
Then clearly $ S_i, ~i=1,2,3,4, $ are connected subsets of $ S_{\mathbb{X}} $ and by the construction of $ S_i $ we have, $ x \in S_1, ~ y \in S_2, ~ -x \in S_3 $ and $ -y \in S_4. $ Also, $ S_i \cap S_j = \phi $ for all $ i,j \in \{1,2,3,4\} $ with $ i \neq j. $ As $ S_{\mathbb{X}} \setminus \{ \pm \frac{x+y}{\|x+y\|}, \pm \frac{x-y}{\|x-y\|} \} = \bigcup \limits_{i=1}^{4} S_i, $ $ M_T \subseteq \bigcup \limits_{i=1}^{4} S_i. $ Also for each disjoint connected set $ S_i, $ $ S_i \cap M_T \neq \phi. $ Therefore, $ M_T $ must have more than two components.
Hence using Theorem \ref{th:BS-property}, we conclude that $ T $ does not satisfy the B\v{S} Property.
\end{proof}
The following example illustrates the applicability of Theorem \ref{th:BS-property} in studying the B\v{S} Property of a bounded linear operator between Banach spaces.
\begin{example} \label{example-1}
Consider a bounded linear operator $ T : \ell_{1}^{2} \to \ell_{\infty}^{3}, $ defined by $ T(x,y) = (x,y,\frac{x+y}{2}) $ for all $ (x,y) \in \ell_{1}^{2}. $ Then it is easy to see that $ \|T\|=1 $ and $ M_T =\{\pm(1,0), \pm(0,1) \}. $ Therefore by using Theorem \ref{th:BS-property}, we conclude that $ T $ does not satisfy the B\v{S} Property. 
\end{example}
If $ \mathbb{X} $ is a two-dimensional Banach space, then from Theorem $ 2.1 $ of \cite{SP} and Theorem $ 2.3 $ of \cite{SPH}, it follows that $ T \in \mathbb{L}(\mathbb{X}) $ satisfies the B\v{S} Property if and only if $ M_T = D \cup (-D), $ where $ D $ is a connected subset of $ S_{\mathbb{X}}. $ If $ \dim \mathbb{X} \geq 3 $ and $ M_T $ has more than two components then it is not known whether $ T $ will satisfy the B\v{S} Property. Our next result gives some insight in that direction, under certain assumptions on the co-domain space $ \mathbb{Y} $ and the norm attainment set $ M_T, $ for a bounded linear operator $ T \in \mathbb{L}(\mathbb{X}, \mathbb{Y}). $

\begin{theorem} \label{th:BS-n-dim}
Let $ \mathbb{X} $ be an $n$-dimensional Banach space, where $ n \geq 3 $ and let $ \mathbb{Y} $ be any Banach space. Let $ T \in \mathbb{L}(\mathbb{X},\mathbb{Y}), $ with $ |M_T| \geq 4, $ be such that the following conditions are satisfied:\\
(a) There exists a basis $ \{x_1, x_2, x_3, \ldots, x_n \} $ of $ \mathbb{X} $ such that $ x_1, x_2 \in M_T. $\\
(b) There exist scalars $ \alpha_3, \alpha_4, \ldots, \alpha_n $ and $ \beta_3, \beta_4, \ldots, \beta_n $ such that for each $ w = c_{1w}x_1+c_{2w}x_2+\ldots+c_{nw}x_n \in M_T, $ 
we have, 
\[ c_{1w}+c_{2w}+\alpha_3 c_{3w}+ \ldots + \alpha_n c_{nw} \neq 0~ \textit{and}~  c_{1w}-c_{2w}+\beta_3 c_{3w}+ \ldots + \beta_n c_{nw} \neq 0. \]
Then at least one of the following is true:\\
$(i)$ $ \bigcup \limits_{x\in M_T} (Tx)^{\bot} = \mathbb{Y}. $\\
$(ii)$ $ T $ does not satisfy the B\v{S} Property.			
\end{theorem}
\begin{proof}
Assuming $ (i) $ is not true, we show that $ T $ does not satisfy the B\v{S} Property.  Under this assumption, $ \bigcup \limits_{x\in M_T} (Tx)^{\bot} \subsetneqq \mathbb{Y}. $ Let us take $ z \in \mathbb{Y} \setminus \bigcup \limits_{x\in M_T} (Tx)^{\bot}. $ We note that it follows from Proposition $ 2.1 $ of \cite{Sa} that for each $ i=1,2, $ either $ z \in (Tx_i)^{+} $ or $ z \in (Tx_i)^{-}. $ \\

\textbf{Case I:} Let $ z \in (Tx_1)^{+} \cap (Tx_2)^{+} $ or $ z \in (Tx_1)^{-} \cap (Tx_2)^{-}. $ Let us define $ A : \mathbb{X} \to \mathbb{Y} $ by
 $$ Ax_1=z, Ax_2= -z ~\mbox{and}~  Ax_i = \beta_i z ~\mbox{for} ~ i = 3, 4, \ldots, n. $$
  If $ z \in (Tx_1)^{+} \cap (Tx_2)^{+}, $ then $ Ax_1 \in (Tx_1)^{+} $ and $ Ax_2 \in (Tx_2)^{-}. $ On the other hand, if $ z \in (Tx_1)^{-} \cap (Tx_2)^{-}, $ then $ Ax_1 \in (Tx_1)^{-} $ and $ Ax_2 \in (Tx_2)^{+}. $ Therefore, using Theorem $ 2.2 $ of \cite{Sa}, we have $ T \bot_{B} A, $ in both the cases. We claim that $ Tu \not\perp_{B} Au $ for any $ u \in M_T. $ Let $ u = c_{1u}x_1+c_{2u}x_2+\ldots+c_{nu}x_n \in M_T. $ Then, $ Au = (c_{1u}-c_{2u}+c_{3u}\beta_3+\ldots+c_{nu}\beta_n)z = \gamma z, $ (say). As $ \gamma \neq 0 $ and $ Tu \not\perp_{B} z, $ so we conclude that $ Tu \not\perp_{B} Au. $ Thus $ T $ does not satisfy the B\v{S} Property.\\  
  
\textbf{Case II:} Let $ z \in (Tx_1)^{+} \cap (Tx_2)^{-} $ or $ z \in (Tx_1)^{-} \cap (Tx_2)^{+}. $ Let us define $ A : \mathbb{X} \to \mathbb{Y} $ by $ Ax_1=z, $ $ Ax_2= z $ and $ Ax_i = \alpha_i z $ for $ i = 3, 4, \ldots, n. $ Proceeding in a similar manner, we can conclude that $ T \bot_{B} A $ but there exists no $ u \in M_T $ such that $ Tu \bot_{B} Au. $ Therefore $ T $ does not satisfy the B\v{S} Property. This completes the proof of the theorem.
\end{proof}

The above theorem indirectly hints at the importance of the notion of Property $ P_n $ for a Banach space in the study of the B\v{S} Property of a bounded linear operator. We state this formally in the following corollary. 

\begin{cor} \label{cor:BS-n-dim}
Let $ \mathbb{X} $ be an $n$-dimensional Banach space, where $ n \geq 3 $ and let $ \mathbb{Y} $ be a Banach space such that $ \mathbb{Y} $ has Property $ P_m,$ for some $ m \in \mathbb{N}. $ Let $ T \in \mathbb{L}(\mathbb{X},\mathbb{Y}) $ be such that the following conditions are satisfied:\\
(a) $ |M_T| \geq 4 $ and $ |T(M_T)| \leq 2m, $ \\ 
(b) There exists a basis $ \{x_1, x_2, x_3, \ldots, x_n \} $ of $ \mathbb{X} $ such that $ x_1, x_2 \in M_T, $\\
(c) There exist scalars $ \alpha_3, \alpha_4, \ldots, \alpha_n $ and $ \beta_3, \beta_4, \ldots, \beta_n $ such that for each $ w = c_{1w}x_1+c_{2w}x_2+\ldots+c_{nw}x_n \in M_T, $
we have that
\[  c_{1w}+c_{2w}+\alpha_3 c_{3w}+ \ldots + \alpha_n c_{nw} \neq 0~\textit{and}~ c_{1w}-c_{2w}+\beta_3 c_{3w}+ \ldots + \beta_n c_{nw} \neq 0. \]
Then $ T $ does not satisfy the B\v{S} Property.
\end{cor}
\begin{proof}
As $ T(M_T) $ contains at most $ 2m $ elements which are pairwise scalar multiples of one another and the co-domain space $ \mathbb{Y} $ has Property $ P_m, $ we must have $ \bigcup \limits_{x\in M_T} (Tx)^{\bot} \subsetneqq \mathbb{Y}. $ Then from Theorem \ref{th:BS-n-dim}, we conclude that $ T $ does not satisfy the B\v{S} Property.
\end{proof}

We now give an example to illustrate the applicability of the Corollary \ref{cor:BS-n-dim} in studying the B\v{S} Property of a bounded linear operator between Banach spaces. Here we would like to mention that every smooth Banach space of dimension at least $ 2, $ has Property $ P_n, $ for each $ n \in \mathbb{N}. $
\begin{example} \label{example-2}
	Consider a bounded linear operator $ T : \ell_{\infty}^{3} \to \ell_{2}^{3}, $ defined by $ Tx= \frac{x}{\sqrt{3}} $ for all $ x \in \ell_{\infty}^{3}. $ Then it is easy to see that $ \|T\|=1, $ $M_T= \{ \pm(1,1,1), \pm(-1,1,1), $ \\
$	\pm(-1,-1,1), \pm(1,-1,1) \} $ and $ T(M_T) = \{ \pm(\frac{1}{\sqrt{3}}, \frac{1}{\sqrt{3}}, \frac{1}{\sqrt{3}}), \pm(\frac{-1}{\sqrt{3}}, \frac{1}{\sqrt{3}}, \frac{1}{\sqrt{3}}), $ \\
$\pm(\frac{-1}{\sqrt{3}}, \frac{-1}{\sqrt{3}}, \frac{1}{\sqrt{3}}), \pm(\frac{1}{\sqrt{3}}, \frac{-1}{\sqrt{3}}, \frac{1}{\sqrt{3}}) \}. $ Consider $ x_1= (1,1,1), x_2= (-1,1,1), x_3= $ \\ $(-1,-1,1) $ and $ x_4= (1,-1,1). $ Clearly $ \{x_1,x_2,x_3\} $ forms a basis of $ \ell_{\infty}^{3}. $ If we choose $ \alpha= \beta=1, $ then condition (c) of Corollary \ref{cor:BS-n-dim} is satisfied. Also, $ \ell_{2}^{3} $ has Property $ P_n, $ for any $ n \in \mathbb{N}, $ as $ \ell_{2}^{3} $ is a smooth space. Therefore, by using Corollary \ref{cor:BS-n-dim}, we conclude that $ T $ does not satisfy the B\v{S} Property.
\end{example}

It is worth mentioning that Theorem $ 2.2 $ of \cite{SPH} holds true when the domain space is any finite-dimensional Banach space and the co-domain space is any smooth Banach space of dimension at least two. Indeed, the said theorem can be stated in the following more general form, by using the same arguments presented in the proof of the original result.
\begin{theorem}
Let $ \mathbb{X} $ be a finite-dimensional Banach space and $ \mathbb{Y} $ be a smooth Banach space of dimension greater than or equal to two. Let $ T \in \mathbb{L}(\mathbb{X}, \mathbb{Y}) $ be such that $ M_T $ is a countable set with more than two points. Then $ T $ does not satisfy the B\v{S} Property.
\end{theorem} 

 In the remaining part of this article, we focus on Property $ P_n $ for polyhedral Banach spaces. Our first observation reveals that given any polyhedral Banach space $ \mathbb{X}, $ there exists a natural number $ n_0 $ such that $ \mathbb{X} $ does not have Property $ P_n, $ for any $ n \geq n_0. $
 
\begin{theorem} \label{th:P_n-property}
Let $\mathbb{X}$ be a finite-dimensional polyhedral Banach space such that $ B_{\mathbb{X}} $ has exactly $ 2n $ extreme points. Then $ \mathbb{X} $ does not have Property $ P_n. $
\end{theorem}

\begin{proof}
Let us denote the extreme points of $ B_{\mathbb{X}} $ by $ \pm u_1, \pm u_2, \ldots, \pm u_n. $ We claim that $ \bigcup \limits_{i=1}^{n} u_{i}^{\bot} = \mathbb{X}. $ Let $ y \in \mathbb{X} $ be arbitrary. Given $ z \in \mathbb{X} $ there exists a scalar $ a \in \mathbb{R} $ such that $ ay+z \bot_{B} y, $ by Theorem $ 2.3 $ of \cite{J}. Take $ x= \frac{ay+z}{\|ay+z\|}, $ then $ x \bot_{B} y. $ If $ x $ is an extreme point of $ B_{\mathbb{X}}, $ then we have nothing more to show. Now, suppose that $ x $ is not an extreme point of $ B_{\mathbb{X}}. $ As $ x \bot_{B} y, $ by using Theorem $ 2.1 $ of \cite{J}, there exists a linear functional $ f \in S_{\mathbb{X}^*} $ such that $ f(x)=\|x\|=1 $ and $ f(y)=0. $ Since $ f $ attains norm, it is easy to see that there exists an extreme point $ u_i $ of $ B_{\mathbb{X}} $ such that $ |f(u_i)| = \|f\|=1. $ Therefore,  by Theorem $ 2.1 $ of \cite{J}, we have $ u_i \bot_{B} y. $ Thus $ \bigcup \limits_{i=1}^{n} u_{i}^{\bot} = \mathbb{X}. $ This completes the proof of the theorem.
\end{proof}

Our next theorem shows that we have a definitive answer for two-dimensional polyhedral Banach spaces, regarding Property $ P_n. $ To prove the theorem, we need the following lemma:

\begin{lemma}\label{lemma:orthogonal}
Let $ \mathbb{X} $ be a two-dimensional polyhedral Banach space. Then for any $ x \in E_{\mathbb{X}}, $ there exists a normal cone $ K $ of $ \mathbb{X} $ such that $ x^{\bot} = K \cup (-K). $ In addition, if the normal cone $ K $ is determined by $ v_1, v_2 \in S_{\mathbb{X}}, $ then $ \{ (1-t) v_1 + t v_2 : t \in (0,1) \} \cap y^{\bot} = \phi $ for each $ y \in E_{\mathbb{X}} \setminus \{ \pm x \}. $
\end{lemma}
\begin{proof}
Let $ g \in S_{\mathbb{X}^*} $ and $ g(x) = \|x\|=1, $ i.e., $ g $ is a supporting functional of $ B_{\mathbb{X}} $ at $ x. $ Let $ f_{1} $ and $ f_{2} $ be the two supporting functionals corresponding to the two edges of $ S_{\mathbb{X}} $ meeting at $ x. $ Now, $ x + \ker g $ is a supporting line to $ B_{\mathbb{X}} $ at $ x, $ that lies entirely within the cone formed by the straight lines $ x + \ker f_{1} $ and $ x + \ker f_{2}. $ For $ i=1,2, $ let 
\[ f_{i}^{+} = \{ z \in \mathbb{X} : f_{i}(z) \geq 0\} ~\mbox{and}~  f_{i}^{-} = \{ z \in \mathbb{X} : f_{i}(z) \leq 0\}. \]
\begin{figure}[ht]
	\centering 
	\includegraphics[width=0.9\linewidth]{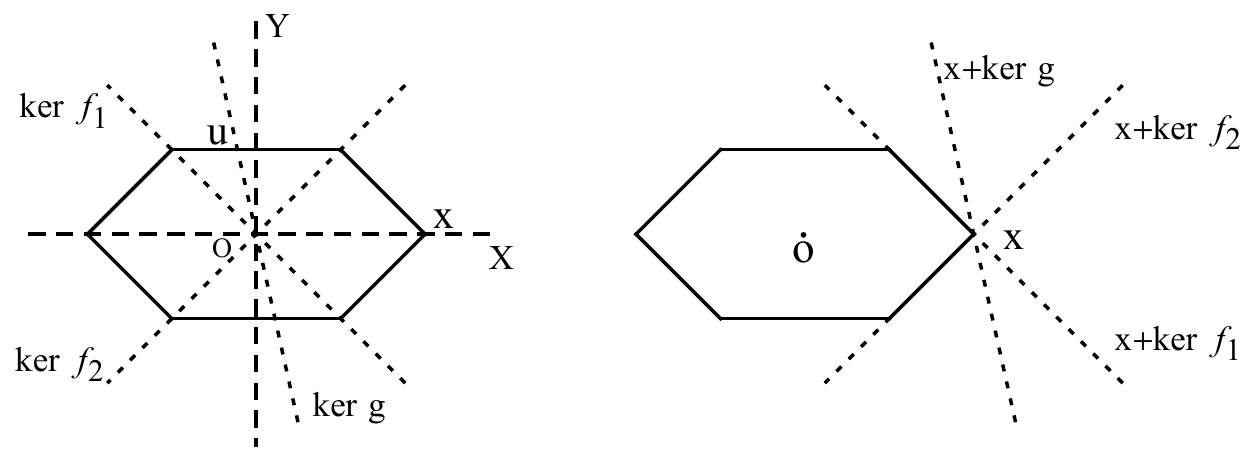}
	\caption{}
	\label{figure}
\end{figure}
We note that each $ f_{i}^{+}~(f_{i}^{-}) $ is a closed half-space in $ \mathbb{X}. $ Let $ u \in \ker g $ be arbitrary. The situation is illustrated in Figure $ 1. $ It follows immediately  that either of the following must be true:\\

$ (i) $ $ u \in f_{1}^{+} $ and $ u \in f_{2}^{-},~ (ii)~ u \in f_{1}^{-} $ and $ u \in f_{2}^{+}. $\\ 

Taking $ K = f_{1}^{+} \cap f_{2}^{-}, $ it is easy to see that $ -K = f_{1}^{-} \cap f_{2}^{+}. $ Thus for any $ u \in \mathbb{X}, $ with $ x \bot_{B} u, $ we have $ u \in K \cup (-K). $ Therefore, $ x^{\bot} = \{ w \in \mathbb{X} : x \perp_B w \} = K \cup (-K). $  This completes the proof  of the first part of the lemma.\\
 Next, suppose $ K $ is determined by $ v_{1}, v_{2} \in S_{\mathbb{X}}. $ From the construction of $ K $ it is clear that $ v_{1} \in \ker f_{1} \cap S_{\mathbb{X}} $ and $ v_{2} \in \ker f_{2} \cap S_{\mathbb{X}}. $ Let $ V = \{(1-t) v_1 + t v_2 : t \in (0,1) \}. $ We show that $ V \cap y^{\bot} = \phi, $ for each $ y \in E_{\mathbb{X}} \setminus \{ \pm x \}. $ If possible, suppose that $ V \cap y^{\bot} \neq \phi, $ for some $ y \in E_{\mathbb{X}} \setminus \{ \pm x \}. $ Then there exists $ v = (1-t) v_1 + t v_2 \in V $ such that $ v \in y^{\bot}. $  Since $ x \perp_B v, $ there exists $ f \in S_{\mathbb{X}^*} $ such that $ f(x)= \|x\|=1 $ and $ f (v) = 0. $ On the other hand, since $ y \perp_B v, $ there exists $ h \in S_{\mathbb{X}^*} $ such that $ h(y)= \|y\|= 1 $ and $ h (v) = 0. $ Since $ \mathbb{X} $ is two-dimensional, it is easy to deduce that $ f = \pm h. $  Therefore, either $ x, y $ are adjacent vertices and $ f = h $ is an extreme supporting functional corresponding to the facet $L[x,y] = \{ (1-t)x + ty : t \in [0,1] \}, $ or, $ x, -y $ are adjacent vertices and $ f = -h $ is an extreme supporting functional corresponding to the facet $L[x,-y] = \{ (1-t)x + t(-y) : t \in [0,1] \}. $ As $ f_{1} $ and $ f_{2} $ are two supporting functionals corresponding to the two edges of $ S_{\mathbb{X}} $ meeting at $ x, $ $ f $ is equal to either $ f_{1} $ or $ f_{2}. $ Therefore, $ v $ is equal to either $ v_{1} $ or $ v_{2}, $ which is a contradiction to our assumption that $ v \in V. $ This completes the proof of the lemma.	
\end{proof}
We now prove the desired theorem.
\begin{theorem}\label{th:P_(n-1)-property}
Let $ \mathbb{X} $ be a two-dimensional polyhedral Banach space such that $ B_{\mathbb{X}} $ has exactly $ 2n $ extreme points for some $ n \in \mathbb{N}. $ Then $ \mathbb{X} $ has Property $ P_{n-1}. $
\end{theorem}

\begin{proof}
If $ x \in S_{\mathbb{X}} $ is a non-extreme point of $ B_{\mathbb{X}}, $ then there exist $ x_1, x_2 \in S_{\mathbb{X}} $ such that $ x = (1-t)x_1 + t x_2 $ for some $ t \in (0,1) $ and $ x_1, x_2 $ are extreme points of $ B_{\mathbb{X}}. $ It is easy to observe that $ x^{\bot} \subsetneqq x_{1}^{\bot} $ and $ x^{\bot} \subsetneqq x_{2}^{\bot}. $ Therefore, without loss of generality we may consider any $ (n-1) $ extreme points of $ B_{\mathbb{X}}, $ instead of any $ (n-1) $ points in $ S_{\mathbb{X}}, $ to prove that $ \mathbb{X} $ has Property $ P_{n-1}. $\\
Let $ E_{\mathbb{X}} = \{\pm x_1, \pm x_2, \ldots, \pm x_n \}.$  Then from Lemma \ref{lemma:orthogonal}, for each $ i \in \{1,2, \ldots, n\}, $ there exists normal cone $ K_{i} $ of $ \mathbb{X} $ such that $ x_{i}^{\bot} = K_{i} \cup (-K_{i}). $ If the normal cone $ K_{i} $ is determined by $ v_{i1}, v_{i2} \in S_{\mathbb{X}}, $ then $ \{ (1-t) v_{i1} + t v_{i2} : t \in (0,1) \} \cap\bigcup \limits_{\substack{j=1 \\j \neq i}}^{n} x_{j}^{\bot} = \phi. $ Since $ x_i $ is any extreme point of $ B_{\mathbb{X}}, $ we conclude that the union of the Birkhoff-James orthogonality sets of any $ (n-1) $ extreme points of $ B_{\mathbb{X}} $ must be a proper subset of $ \mathbb{X}. $ However, this is clearly equivalent to the fact that $ \mathbb{X} $ has Property $ P_{n-1}. $ This establishes the theorem.  
\end{proof}

As an application of Corollary \ref{cor:BS-n-dim}, we next give an example of a bounded linear operator $ T $ between a three-dimensional polyhedral Banach space $ \mathbb{X} $ and a two-dimensional polyhedral Banach space $ \mathbb{Y} $ such that $ T $ does not satisfy the B\v{S} Property. 

\begin{example} \label{example-3}
	Let $ \mathbb{X} = \ell_{\infty}^{3} $ and let $ \mathbb{Y} $ be a two-dimensional real polyhedral Banach space such that $ S_{\mathbb{Y}} $ is regular decagon with vertices $ (\cos\frac{j\pi}{5}, \sin\frac{j\pi}{5}), $ $ j \in \{0,1,2, \ldots, 9\}. $ Consider the linear operator $ T : \mathbb{X} \to \mathbb{Y}, $ defined by 
	$$ T (x,y,z)= \Big(\frac{x+y}{2}+ \frac{(y-x)\cos\frac{2\pi}{5}}{2}, \frac{(y-x)\sin\frac{2\pi}{5}}{2} \Big) . $$
It is easy to check that $ \|T\|=1, $ $M_T= \{ \pm(1,1,z), \pm(-1,1,z) : z \in [-1,1] \} $ and $ T(M_T) = \{ \pm(1,0), \pm(\cos\frac{2\pi}{5}, \sin\frac{2\pi}{5}) \}. $ Consider $ x_1= (1,1,1), x_2= (-1,1,1) $ and $ x_3=(-1,-1,1). $ Clearly $ \{x_1,x_2,x_3\} $ forms a basis of $ \mathbb{X}. $ If we choose $ \alpha= -10 $ and $ \beta= \frac{-3}{2}, $ then condition (c) of Corollary \ref{cor:BS-n-dim} is satisfied. From Theorem \ref{th:P_(n-1)-property}, we know that $ \mathbb{Y} $ has Property $ P_4. $ Therefore, by using Corollary \ref{cor:BS-n-dim}, we conclude that $ T $ does not satisfy the B\v{S} Property.
\end{example}

For any two Banach spaces $ \mathbb{X}, \mathbb{Y}, $ it is easy to see that $ \mathbb{X} \times \mathbb{Y}, $ equipped with the norm $ \|(x,y)\|= \max \{ \|x\|, \|y\| \} $ for all $ (x,y) \in \mathbb{X} \times \mathbb{Y}, $ is a Banach space. Let us denote this space by $\mathbb{X} \oplus_{\infty} \mathbb{Y}. $ Similarly, $ \mathbb{X} \times \mathbb{Y}, $ equipped with the norm $ \|(x,y)\|=  \|x\|+ \|y\|  $ for all $ (x,y) \in \mathbb{X} \times \mathbb{Y}, $ is a Banach space which is denoted by $\mathbb{X} \oplus_{1} \mathbb{Y}. $ In the following theorems, we study Property $ P_n $ for Banach spaces $\mathbb{X} \oplus_{\infty} \mathbb{Y} $ and $\mathbb{X} \oplus_{1} \mathbb{Y}, $ where $ \mathbb{X} $ is a polyhedral Banach space.

\begin{theorem} \label{th:infinity-sum}
Let $ \mathbb{X} $ be a polyhedral Banach space such that $ \mathbb{X} $ does not have Property $P_n,$ for some $ n \in \mathbb{N}. $ Then $\mathbb{X} \oplus_{\infty} \mathbb{Y} $ does not have Property $ P_n, $ for any Banach space $ \mathbb{Y}. $ Moreover, if $ \mathbb{Y} $ is a polyhedral Banach space such that $ \mathbb{Y} $ does not have Property $ P_m, $ for some $ m \in \mathbb{N}, $ then $\mathbb{X} \oplus_{\infty} \mathbb{Y} $ does not have Property $ P_r, $ where $ r = \min \{ m,n \}. $ 
\end{theorem}
\begin{proof}
As $ \mathbb{X} $ does not have Property $ P_n, $ there exist $ x_1, x_2, \ldots, x_n \in S_{\mathbb{X}} $ such that $ \bigcup \limits_{i=1}^{n} x_{i}^{\bot} = \mathbb{X}. $ Now we claim that $ \bigcup \limits_{i=1}^{n} (x_{i}, 0)^{\bot} = \mathbb{X} \oplus_{\infty} \mathbb{Y}. $\\
Let us first show that $ x_{i}^{\bot} \times \mathbb{Y} \subseteq (x_i,0)^{\bot}, $ for each $ i \in \{1,2, \ldots, n \}. $ Let $ (x,y) \in x_{i}^{\bot} \times \mathbb{Y}. $ Then for any scalar $ \lambda, $
\begin{eqnarray*}
\|(x_i,0) + \lambda (x,y)\| & = & \|(x_i + \lambda x, \lambda y) \| \\
                            & = & \max \{ \|x_i + \lambda x\|, \|\lambda y\| \} \\
														& \geq & \|x_i + \lambda x\| \\
														& \geq & \|x_i\| = \|(x_i,0)\|,
\end{eqnarray*} 
as $ x \in x_{i}^{\bot}. $ Therefore, $ (x,y) \in (x_i,0)^{\bot}. $ Hence $ x_{i}^{\bot} \times \mathbb{Y} \subseteq (x_i,0)^{\bot}. $\\
Therefore, $ \bigcup \limits_{i=1}^{n} (x_{i}, 0)^{\bot} \supseteq \bigcup \limits_{i=1}^{n} (x_{i}^{\bot} \times \mathbb{Y}) = \mathbb{X} \oplus_{\infty} \mathbb{Y}. $ Hence $ \bigcup \limits_{i=1}^{n} (x_{i}, 0)^{\bot} = \mathbb{X} \oplus_{\infty} \mathbb{Y}. $\\
Further if $ \mathbb{Y} $ does not have Property $ P_m, $ then as before we can show that $\mathbb{X} \oplus_{\infty} \mathbb{Y} $ does not possess Property $ P_m. $ Thus $\mathbb{X} \oplus_{\infty} \mathbb{Y}$ does not have Property $ P_r, $ where $ r = \min \{ m, n \}. $ This completes the proof of the theorem.
\end{proof}
\begin{cor} \label{cor:l_infity}
Let $ \mathbb{X}= \ell^{n}_{\infty}, $ for any $ n ~(\geq 2) \in \mathbb{N}. $ Then $ \mathbb{X} $ does not have Property $ P_2. $  
\end{cor}
\begin{proof}
From Theorem \ref{th:P_n-property}, it follows that $ \ell^{2}_{\infty} $  does not have Property $ P_2. $ Also, we know that $ \ell^{3}_{\infty} = \ell^{2}_{\infty} \oplus_{\infty} \mathbb{R}. $ Therefore, by using Theorem \ref{th:infinity-sum}, we conclude that $ \ell^{3}_{\infty} $ does not have Property $ P_2. $ Continuation of this argument proves that $ \ell^{n}_{\infty}, $ for any $ n ~(\geq 2) \in \mathbb{N} $ does not have Property $ P_2. $
\end{proof}

Applying similar arguments, the proofs of the following results are now apparent: 

\begin{theorem} \label{th:one-sum}
Let $ \mathbb{X} $ be a polyhedral Banach space such that $ \mathbb{X} $ does not have Property $P_n,$ for some $ n \in \mathbb{N}. $ Then $\mathbb{X} \oplus_{1} \mathbb{Y} $ does not have Property $ P_n, $ for any Banach space $ \mathbb{Y}. $ Moreover, if $ \mathbb{Y} $ is a polyhedral Banach space such that $ \mathbb{Y} $ does not have Property $ P_m, $ for some $ m \in \mathbb{N}, $ then $\mathbb{X} \oplus_{1} \mathbb{Y} $ does not have Property $ P_r, $ where $ r = \min \{ m,n \}. $ 
\end{theorem}

\begin{cor}
Let $ \mathbb{X}= \ell^{n}_{1}, $ where $ n (\geq 2) \in \mathbb{N}. $ Then $ \mathbb{X} $ does not have Property $ P_2. $  
\end{cor}

Let $ \mathbb{X} $ be a three-dimensional polyhedral Banach space such that  $ B_{\mathbb{X}} $ is a prism with vertices $(\cos\frac{j\pi}{n}, \sin\frac{j\pi}{n},\pm{1}),$  $j \in \{0,1,2,\ldots,2n-1\}, $ $ n \geq 2. $ Then it is trivial to see that $ \mathbb{X}= \mathbb{Y} \oplus_{\infty} \mathbb{R}, $ where $ \mathbb{Y} $ is a two-dimensional polyhedral Banach space so that the extreme points of $ B_{\mathbb{Y}} $ are given by $(\cos\frac{j\pi}{n}, \sin\frac{j\pi}{n}),$  $j \in \{0,1,2,\ldots,2n-1\}, $ $ n \geq 2. $ Therefore, by using Theorem \ref{th:P_n-property} and Theorem \ref{th:infinity-sum}, we can conclude that $ \mathbb{X} $ does not have Property $ P_n. $ However, in the next theorem we show that $ \mathbb{X} $ does not have Property $ P_2, $ for any $ n \geq 2. $

\begin{theorem} \label{th:prism}
Let $ \mathbb{X} $ be a three-dimensional polyhedral Banach space such that  $ B_{\mathbb{X}} $ is a prism with vertices $(\cos\frac{j\pi}{n}, \sin\frac{j\pi}{n},\pm{1}),$  $j \in \{0,1,2,\ldots,2n-1\}, $ $ n \geq 2. $ Then $ \mathbb{X} $ does not have Property $ P_2. $
\end{theorem}
\begin{proof}
Let the vertices of $B_{\mathbb{X}}$ be  $v_{\pm(j+1)}$, $j \in \{0,1,2,\ldots,2n-1\}, $ where $v_{\pm(j+1)}=(\cos\frac{j\pi}{n},\sin\frac{j\pi}{n},\pm{1}).$ The unit sphere $S_{\mathbb{X}}$ is shown in Figure $2.$\\
\begin{figure}[ht]
\centering 
\includegraphics[width=0.5\linewidth]{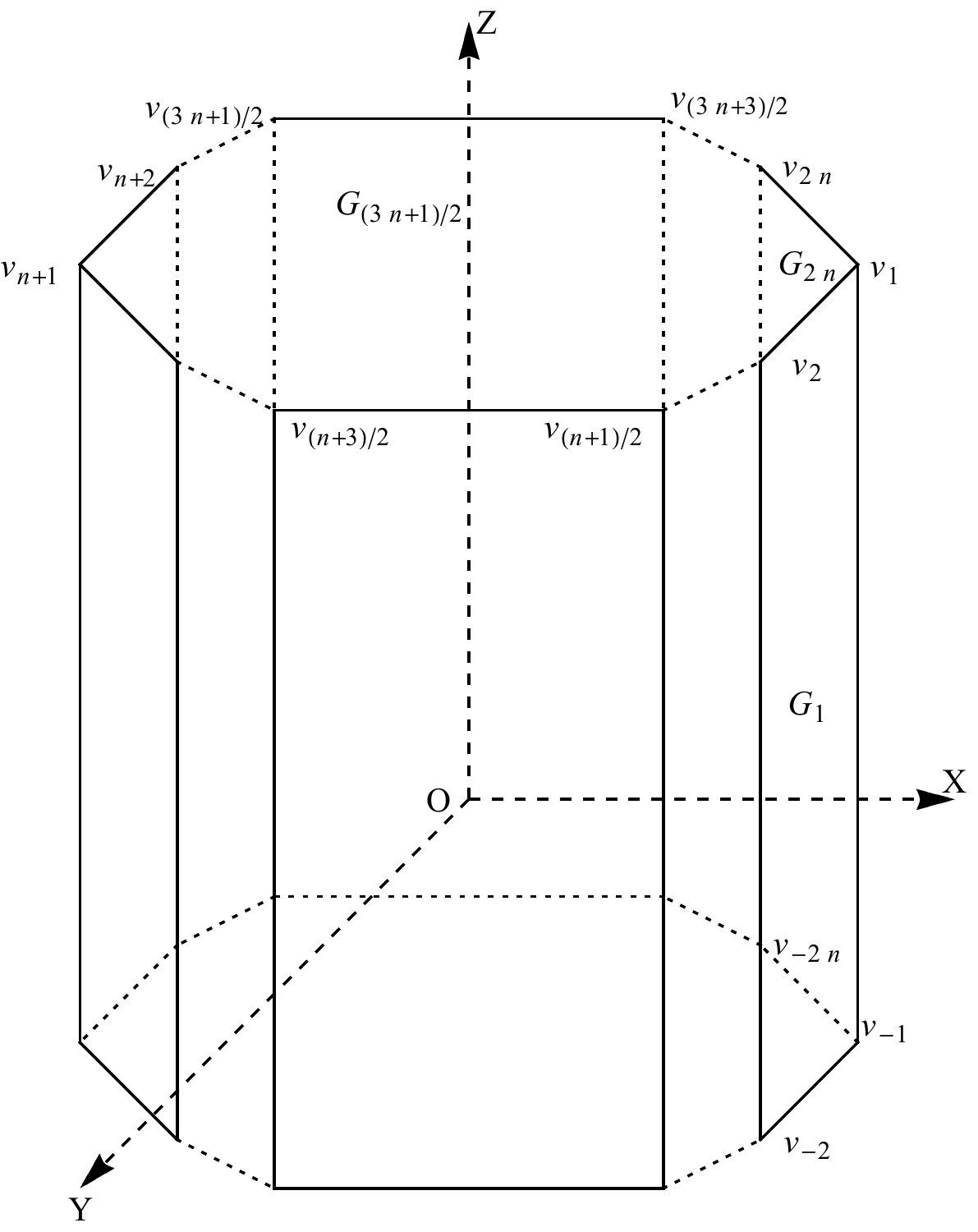}
\caption{}
\label{figure}
\end{figure}
A simple computation reveals the explicit expression for the norm function on $ \mathbb{X}. $ Given any $ (x,y,z) \in \mathbb{X}, $ we have,
 \[ \|(x,y,z)\|= \max \limits_{0 \leq j \leq 2n-1} \left\{ \frac{|\cos \frac{(2j+1) \pi}{2n}| |x|}{\cos \frac{\pi}{2n}} +\frac{|\sin \frac{(2j+1) \pi}{2n}| |y|}{\cos \frac{\pi}{2n}}, |z|  \right\}. \]
We claim that $ v_{1}^{\bot} \cup v_{n+1}^{\bot} = \mathbb{X}. $ Let $ (x,y,z) \in \mathbb{X} $ be such that $ x \geq 0, z \leq 0 $ and $ y $ is arbitrary. Now, for any scalar $ \lambda \geq 0, $
\begin{eqnarray*}
    \|(1,0,1)+\lambda (x,y,z)\| & = & \| (1 +\lambda x, \lambda y, 1+ \lambda z) \| \\
		                            & = & \max \limits_{0 \leq j \leq 2n-1} \Big\{ \frac{|\cos \frac{(2j+1) \pi}{2n}| |1+ \lambda x|}{\cos \frac{\pi}{2n}} +\frac{|\sin \frac{(2j+1) \pi}{2n}| |\lambda y|}{\cos \frac{\pi}{2n}}, \\
																& & \hspace{1.7 cm} |1 + \lambda z|  \Big\} \\
																& \geq & |1+ \lambda x| + \tan \frac{\pi}{2n} |\lambda y|	\\			
										           & \geq & 1 = \|(1,0,1)\|.
\end{eqnarray*}
 Also, for any scalar $ \lambda \leq 0, $
\begin{eqnarray*}
   \|(1,0,1)+\lambda (x,y,z)\| & = & \| (1 +\lambda x, \lambda y, 1+ \lambda z) \| \\
		                            & = & \max_{0 \leq j \leq 2n-1} \Big\{ \frac{|\cos \frac{(2j+1) \pi}{2n}| |1+ \lambda x|}{\cos \frac{\pi}{2n}} +\frac{|\sin \frac{(2j+1) \pi}{2n}| |\lambda y|}{\cos \frac{\pi}{2n}},  \\
																& & \hspace{1.7 cm} |1 + \lambda z| \Big \} \\
																& \geq & |1+ \lambda z|	\\			
										           & \geq & 1 = \|(1,0,1)\|.
\end{eqnarray*}

Therefore, $ (1,0,1) \bot_{B} (x,y,z), $ for all $ x \geq 0, z \leq 0 $ and for any $ y. $ From the homogeneity property of Birkhoff-James orthogonality, it follows that  $ (1,0,1) \bot_{B} (x,y,z), $ for all $ x \leq 0, z \geq 0 $ and for any $ y. $\\
 Let $ (x,y,z) \in \mathbb{X} $ be such that $ x \geq 0, z \geq 0 $. Now for any $ \lambda \geq 0, $
\begin{eqnarray*}
    \|(-1,0,1)+\lambda (x,y,z)\| & = & \| (-1 +\lambda x, \lambda y, 1+ \lambda z) \| \\
		                            & = & \max \limits_{0 \leq j \leq 2n-1} \Big\{ \frac{|\cos \frac{(2j+1) \pi}{2n}| |-1+ \lambda x|}{\cos \frac{\pi}{2n}} +\frac{|\sin \frac{(2j+1) \pi}{2n}| |\lambda y|}{\cos \frac{\pi}{2n}},\\
		                            & & \hspace{1.7 cm} |1 + \lambda z| \Big \} \\
									& \geq & |1+ \lambda z|	\\				
								    & \geq & 1 = \|(-1,0,1)\|.
\end{eqnarray*}
Also, for any $ \lambda \leq 0, $
\begin{eqnarray*}
    \|(-1,0,1)+\lambda (x,y,z)\| & = & \| (-1 +\lambda x, \lambda y, 1+ \lambda z) \| \\
		                            & = & \max \limits_{0 \leq j \leq 2n-1} \Big \{ \frac{|\cos \frac{(2j+1) \pi}{2n}| |-1+ \lambda x|}{\cos \frac{\pi}{2n}} +\frac{|\sin \frac{(2j+1) \pi}{2n}| |\lambda y|}{\cos \frac{\pi}{2n}},\\
		                            & & \hspace{1.7 cm} |1 + \lambda z| \Big \} \\
									& \geq & |-1+ \lambda x| + \tan \frac{\pi}{2n} |\lambda y|	\\			
								    & \geq & 1 = \|(-1,0,1)\|.
\end{eqnarray*}
Therefore, $ (-1,0,1) \bot_{B} (x,y,z), $ for all $ x \geq 0, z \geq 0 $ and for any $ y. $ From the homogeneity property of Birkhoff-James orthogonality, it follows that $ (-1,0,1) \bot_{B} (x,y,z), $ for all $ x \leq 0, z \leq 0 $ and for any $ y. $\\
Hence for any $ (x,y,z) \in \mathbb{X}, $ either $ (x,y,z) \in v_1^{\bot} $ or $ (x,y,z) \in v_{n+1}^{\bot}, $ i.e., $ v_1^{\bot} \cup v_{n+1}^{\bot} = \mathbb{X}. $ Therefore, $ \mathbb{X} $ does not have Property $ P_2. $ This completes the proof of the theorem.
\end{proof}
\begin{remark}
Theorem \ref{th:prism} shows that the space $\mathbb{X} \oplus_{\infty} \mathbb{Y} $ may not have Property $ P_n $ even if one of space has Property $ P_n. $
\end{remark}

Our Previous results point to the fact that there is a large family of three-dimensional polyhedral Banach spaces not having Property $ P_2. $  The concerned spaces have been constructed by taking $ \ell_{\infty} $ sum  of two-dimensional polyhedral Banach spaces with $ \mathbb{R}. $ In the next theorem, we give another such example of a three-dimensional Polyhedral Banach space, not having Property $ P_2, $ which cannot be constructed by taking $ \ell_{\infty} $ sum of lower dimensional Banach space.
 
\begin{theorem} \label{th:prism-pyramid}
Let $\mathbb{X}$  be a three-dimensional polyhedral Banach space such that $B_{\mathbb{X}}$ is a polyhedron obtained by gluing two pyramids at the opposite base faces of a right prism having square base, with vertices $ \pm(1,1,1), \pm(-1,1,1), \pm(-1,-1,1),$ $ \pm(1,-1,1), \pm(0,0,2)$. Then $ \mathbb{X} $ does not have Property $ P_2. $
\end{theorem}

\begin{proof}
Let the vertices of $B_{\mathbb{X}}$ be $v_{\pm j}, j\in \{1,2,3,4\} $ and $w_{\pm 1},$ where $v_{\pm 1}=(1,1,\pm 1)$, $v_{\pm 2}=(-1,1,\pm 1)$, $v_{\pm 3}=(-1,-1,\pm 1)$, $v_{\pm 4}=(1,-1,\pm 1)$ and $w_{\pm 1}=(0,0,\pm 2)$. The unit sphere $S_{\mathbb{X}}$ is shown in Figure $3.$
\medskip

\begin{figure}[ht]
\centering 
\includegraphics[width=0.3\linewidth]{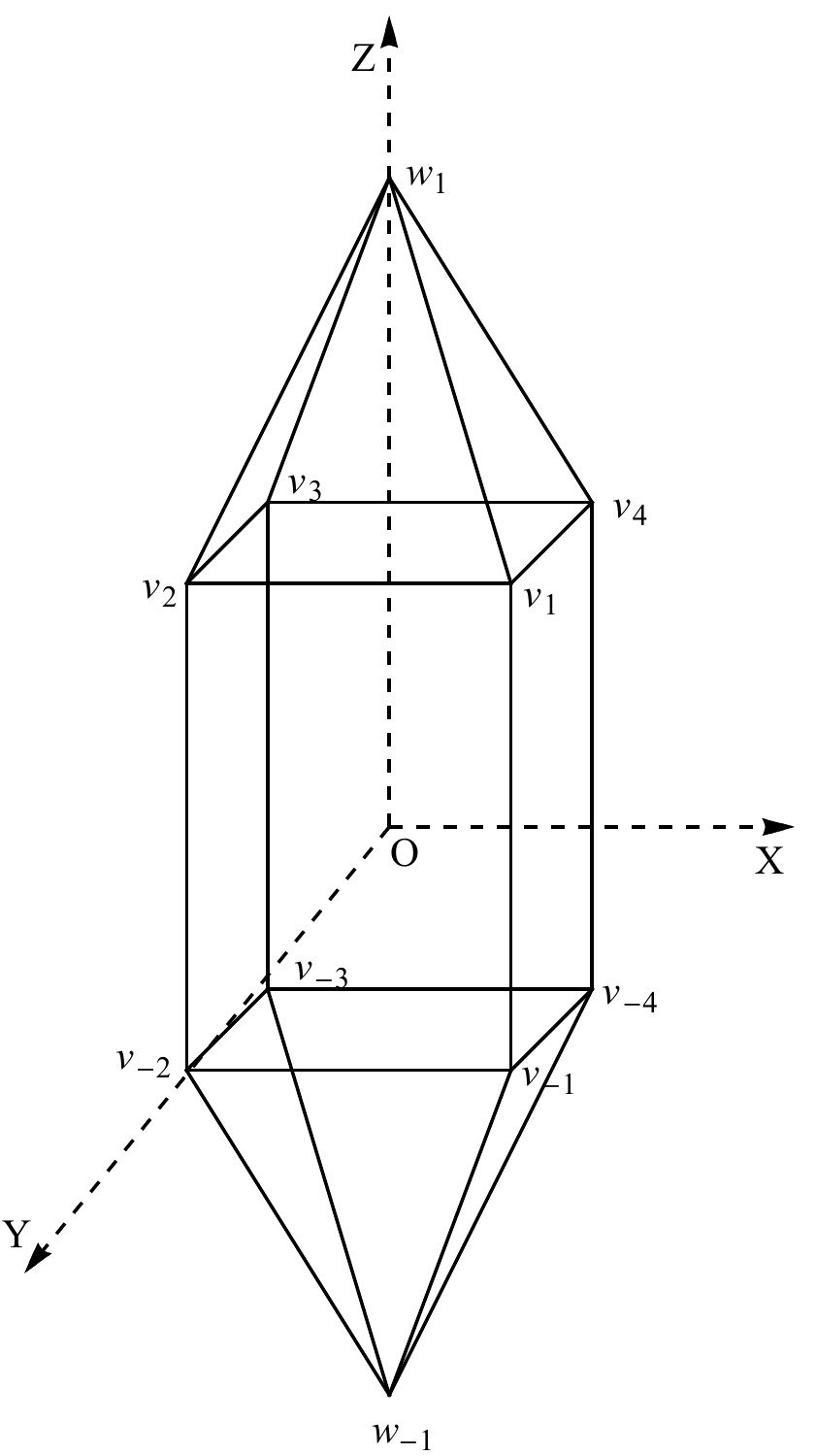}
\caption{}
\label{figure}
\end{figure}

Given any $ (x,y,z) \in \mathbb{X}, $ the expression for the norm function on $ \mathbb{X} $ turns out to be the following:
 \[ \|(x,y,z)\| = \max \left \{ |x|, |y|, \frac{|x|}{2} + \frac{|z|}{2}, \frac{|y|}{2}+ \frac{|z|}{2} \right \}. \]
We claim that, $ v_{1}^{\bot} \cup v_{4}^{\bot} = \mathbb{X}. $ Let $ (x,y,z) \in \mathbb{X} $ be such that $ x \geq 0, y \geq 0. $ Now, for any $ \lambda \geq 0, $
\begin{eqnarray*}
    \|(1,-1,1)+\lambda (x,y,z)\| & = & \| (1 +\lambda x, -1+\lambda y, 1+ \lambda z) \| \\
		                            & = & \max \Big\{  |1+ \lambda x|, |-1+\lambda y|, \frac{|1+ \lambda x|}{2}+\frac{|1 + \lambda z|}{2},\\
		                            & &  \hspace{.9 cm}\frac{|-1+ \lambda y|}{2}+\frac{|1 + \lambda z|}{2} \Big \} \\
																& \geq & |1+ \lambda x|	\\			
										           & \geq & 1 = \|(1,-1,1)\|.
\end{eqnarray*}
 Also, for any $ \lambda \leq 0, $
\begin{eqnarray*}
    \|(1,-1,1)+\lambda (x,y,z)\| & = & \| (1 +\lambda x, -1+\lambda y, 1+ \lambda z) \| \\
		                            & = & \max \Big \{  |1+ \lambda x|, |-1+\lambda y|, \frac{|1+ \lambda x|}{2}+\frac{|1 + \lambda z|}{2},\\
		                            & & \hspace{.9 cm} \frac{|-1+ \lambda y|}{2}+\frac{|1 + \lambda z|}{2} \Big\} \\
																& \geq & |-1+ \lambda y|	\\			
										           & \geq & 1 = \|(1,-1,1)\|.
\end{eqnarray*}
Therefore, $ (1,-1,1) \bot_{B} (x,y,z), $ for all $ x \geq 0, y \geq 0 $ and for any $ z. $ From the homogeneity property of Birkhoff-James orthogonality, it follows that  $ (1,-1,1) \bot_{B} (x,y,z), $ for all $ x \leq 0, y \leq 0 $ and for any $ z. $\\
 Let $ (x,y,z) \in \mathbb{X} $ be such that $ x \geq 0, y \leq 0. $ For any $ \lambda \geq 0, $
\begin{eqnarray*}
    \|(1,1,1)+\lambda (x,y,z)\| & = & \| (1 +\lambda x, 1+\lambda y, 1+ \lambda z) \| \\
		                            & = & \max \Big \{  |1+ \lambda x|, |1+\lambda y|, \frac{|1+ \lambda x|}{2}+\frac{|1 + \lambda z|}{2}, \\
		                            & & \hspace{.9 cm}\frac{|1+ \lambda y|}{2}+\frac{|1 + \lambda z|}{2} \Big\} \\
									& \geq & |1+ \lambda x|	\\			
								    & \geq & 1 = \|(1,1,1)\|.
\end{eqnarray*}
 Also, for any  $ \lambda \leq 0, $
\begin{eqnarray*}
    \|(1,1,1)+\lambda (x,y,z)\| & = & \| (1 +\lambda x, 1+\lambda y, 1+ \lambda z) \| \\
		                            & = & \max \Big\{  |1+ \lambda x|, |1+\lambda y|, \frac{|1+ \lambda x|}{2}+\frac{|1 + \lambda z|}{2}, \\
		                            & & \hspace{.9 cm}\frac{|1+ \lambda y|}{2}+\frac{|1 + \lambda z|}{2} \Big\} \\
										& \geq & |1+ \lambda y|	\\			
										  & \geq & 1 = \|(1,1,1)\|.
\end{eqnarray*}
Therefore, $ (1,1,1) \bot_{B} (x,y,z), $ for all $ x \geq 0, y \leq 0 $ and for any $ z. $ From the homogeneity property of Birkhoff-James orthogonality, it follows that $ (1,1,1) \bot_{B} (x,y,z), $ for all $ x \leq 0, y \geq 0 $ and for any $ z. $\\
Hence for any $ (x,y,z) \in \mathbb{X}, $ either $ (x,y,z) \in v_{1}^{\bot} $ or $ (x,y,z) \in v_{4}^{\bot}, $ i.e., $ v_{1}^{\bot} \cup v_{4}^{\bot} = \mathbb{X}. $ Therefore, $ \mathbb{X} $ does not have Property $ P_2. $ This completes the proof of the theorem.
\end{proof}

We next give an example of a three-dimensional polyhedral Banach space which has Property $ P_2 $ but does not have Property $ P_3. $ 
\begin{theorem} \label{th:prism-pyramid-general}
Let $\mathbb{X}$ be a three-dimensional polyhedral Banach space such that $B_{\mathbb{X}}$ is a polyhedron with vertices $(\cos\frac{j\pi}{n}, \sin\frac{j\pi}{n},\pm{1}), (0,0,\pm 2)$, $ j \in \{0,1,2,\ldots,2n-1\}$, $ n\geq 3$. Then $ \mathbb{X} $ has Property $ P_2 $ but $ \mathbb{X} $ does not have Property $ P_3. $
\end{theorem}

\begin{proof}
We prove the theorem by assuming that $ n $ is an odd integer. Similar calculations hold true, when $ n $ is an even integer.
\begin{figure}[ht]
\centering 
\includegraphics[width=0.5\linewidth]{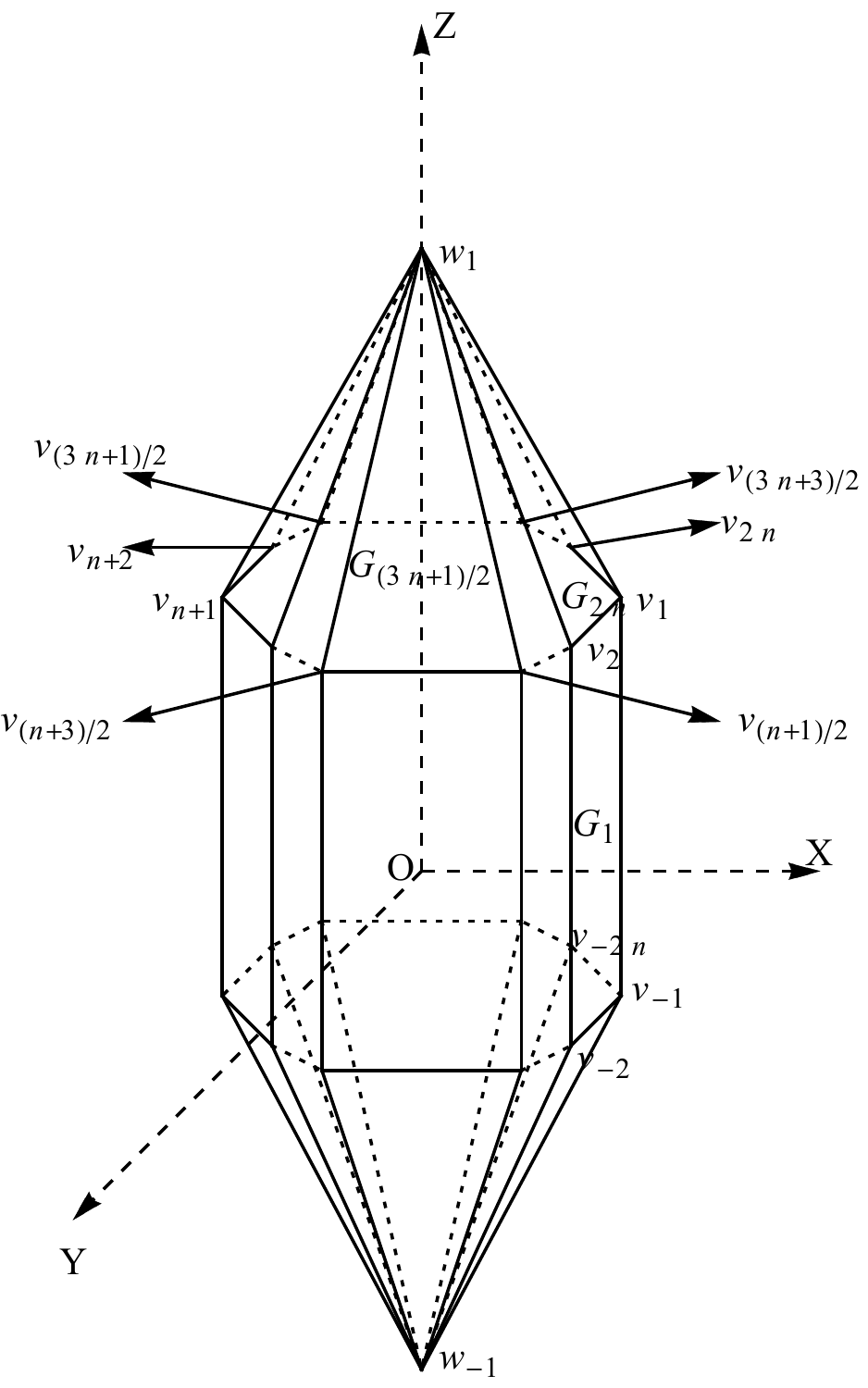}
\caption{}
\label{figure}
\end{figure}
Let the vertices of $B_{\mathbb{X}}$ be  $v_{\pm(j+1)}$, $j \in \{0,1,2,\ldots,2n-1\} $ and $ w_{\pm 1},$  where $v_{\pm(j+1)}=(\cos\frac{j\pi}{n},\sin\frac{j\pi}{n},\pm{1})$ and $w_{\pm 1}=(0,0,\pm 2)$. The unit sphere $S_{\mathbb{X}}$ is shown in Figure $4.$ Let $G_{j+1}$ denote the facet of $B_{\mathbb{X}}$ containing $v_{j+1}, v_{-j-1}, v_{j+2}, v_{-j-2}$, where $v_{2n+1}=v_1$ and $v_{-2n-1}=v_{-1}$. Let $F_{\pm (j+1)}$ denote the facet of $B_{\mathbb{X}}$ containing $v_{\pm (j+1)}, v_{\pm (j+2)},w_{\pm 1}$. For each $j\in \{0,1, 2, \ldots, 2n-1\}$, let $g_{j+1}, f_{\pm (j+1)}$ be the supporting functionals corresponding to the facets $G_{j+1}, F_{\pm (j+1)}$ respectively, i.e., $(v_{j+1}+\ker g_{j+1}) \cap S_{\mathbb{X}} = G_{j+1}$ and $(v_{\pm (j+1)}+\ker f_{\pm (j+1)}) \cap S_{\mathbb{X}} = F_{\pm (j+1)}.$\\
For every $ v_{j+1} \in S_{\mathbb{X}}, ~ j= 0, 1, 2, \ldots, 2n-1, $ there are four adjacent facets $ G_j, G_{j+1},$ $ F_j, F_{j+1} . $ Here we assume that $ G_0= G_{2n} $ and $ F_0=F_{2n}. $ Therefore by Lemma $ 2.1 $ of \cite{SPBB}, extreme supporting functionals corresponding to the vertices $ v_{j+1}, $\\
$  j=0,1,2, \ldots, 2n-1, $ are $ g_j, g_{j+1}, f_j, f_{j+1}. $ Here also we assume that $ g_{0}= g_{2n} $ and $ f_{0}=f_{2n}. $ Now consider the following subsets of $ S_{\mathbb{X}^{*}}: $\\
$ H_{j+1}= \Big\{ h \in S_{\mathbb{X}^{*}} : h = \lambda _1 g_j + \lambda_2 g_{j+1} + \lambda_3 f_j + \lambda_4 f_{j+1}, \lambda_k \geq 0 ~\forall~ k \in \{1,2,3,4\} $\\
 and $ \sum \limits_{k=1}^{4} \lambda_k =1 \Big\}, $
for each $ j=0,1,2, \ldots, 2n-1. $ Then by using Theorem $ 2.1 $ of \cite{J} and Lemma $ 2.1 $ of \cite{SPBB}, we conclude that $ v_{j+1}^{\bot} = \bigcup \limits_{h \in H_{j+1}} \ker h, $ for each $ j=0,1,2, \ldots, 2n-1. $
Again, for $ w_1, $ there are $ 2n $ adjacent facets $ F_{j+1}, ~ j = 0, 1, 2, \ldots, 2n-1. $ Therefore, as before, extreme supporting functionals corresponding to the vertex $ w_1 $ are $ f_{j+1}, ~ j = 0, 1, 2, \ldots, 2n-1. $ Now, consider the following subset of $ S_{\mathbb{X}^{*}}. $\\
$ H = \Big\{ h \in S_{\mathbb{X}^{*}} : h = \sum \limits_{k=0}^{2n-1}\lambda _{k+1} f_{k+1} , \lambda_{k+1} \geq 0 ~\forall~ k \in \{0,1,2,\ldots, 2n-1\} $ \\ and $ \sum \limits_{k=0}^{2n-1} \lambda_{k+1} =1 \Big\}. $
From this, we conclude that $ w_{1}^{\bot} = \bigcup \limits_{h \in H} \ker h. $
Now, for any $ (x,y,z) \in \mathbb{X}, $ we have, \\
 $ \|(x,y,z)\|= \max \limits_{0 \leq j \leq \frac{n-1}{2}} \Big\{ \frac{\cos \frac{(2j+1) \pi}{2n} |x|}{\cos \frac{\pi}{2n}} +\frac{\sin \frac{(2j+1) \pi}{2n} |y|}{\cos \frac{\pi}{2n}}, \frac{\cos \frac{(2j+1) \pi}{2n} |x|}{2\cos \frac{\pi}{2n}} +\frac{\sin \frac{(2j+1) \pi}{2n} |y|}{2\cos \frac{\pi}{2n}}+ \frac{|z|}{2}  \Big\}. $
Using the above expression of the norm function, we can compute the Birkhoff-James orthogonality set of the extreme points of $ B_{\mathbb{X}}. $ For the extreme points of $ B_{\mathbb{X}}, $ lying above the plane $ z = 0, $ we have, \\
 $ v_{1}^{\bot} = (1,0,1)^{\bot} = \pm \Big[ \{(x,y,z) \in \mathbb{X} : x \geq 0, y \geq 0, z \geq 0, x- \tan(\frac{\pi}{2n})y \leq 0\} \\
\cup \{(x,y,z) \in \mathbb{X} : x \leq 0, y \geq 0, z \geq 0, \frac{x}{2}+ \frac{\tan(\frac{\pi}{2n})y}{2} + \frac{z}{2} \geq 0\} \\
\cup \{(x,y,z) \in \mathbb{X} : x \leq 0, y \leq 0, z \geq 0, \frac{x}{2} - \frac{\tan(\frac{\pi}{2n})y}{2} + \frac{z}{2} \geq 0\} \\ \cup \{(x,y,z) \in \mathbb{X} : x \geq 0, y \leq 0, z \geq 0, x+\tan(\frac{\pi}{2n})y \leq 0\} \Big].$\\
Now, for each $ j \in \{1,2,\ldots, \frac{n-1}{2}\}, $ we have\\ 
$ v_{(j+1)}^{\bot}= (\cos\frac{j\pi}{n}, \sin\frac{j\pi}{n},1)^{\bot} = \pm \Big[ \{(x,y,z) \in \mathbb{X} : x \leq 0, y \geq 0, z \geq 0, $ \\ $ \frac{\cos \frac{(2j+1)\pi}{2n}}{2\cos \frac{\pi}{2n}}x + \frac{\sin \frac{(2j+1)\pi}{2n}}{2\cos \frac{\pi}{2n}}y + \frac{z}{2} \geq 0, \frac{\cos \frac{(2j-1)\pi}{2n}}{\cos \frac{\pi}{2n}}x + \frac{\sin \frac{(2j-1)\pi}{2n}}{\cos \frac{\pi}{2n}}y \leq 0\} \\
\cup \{(x,y,z) \in \mathbb{X} : x \leq 0, y \leq 0, z \geq 0, \frac{\cos \frac{(2j+1)\pi}{2n}}{2\cos \frac{\pi}{2n}}x + \frac{\sin \frac{(2j+1)\pi}{2n}}{2\cos \frac{\pi}{2n}}y + \frac{z}{2} \geq 0\} \\
\cup \{(x,y,z) \in \mathbb{X} : x \leq 0, y \leq 0, z \geq 0, \frac{\cos \frac{(2j-1)\pi}{2n}}{2\cos \frac{\pi}{2n}}x + \frac{\sin \frac{(2j-1)\pi}{2n}}{2\cos \frac{\pi}{2n}}y + \frac{z}{2} \geq 0\} \\ 
\cup \{(x,y,z) \in \mathbb{X} : x \geq 0, y \leq 0, z \geq 0, \frac{\cos \frac{(2j-1)\pi}{2n}}{2\cos \frac{\pi}{2n}}x + \frac{\sin \frac{(2j-1)\pi}{2n}}{2\cos \frac{\pi}{2n}}y + \frac{z}{2} \geq 0, $ \\ 
$ \frac{\cos \frac{(2j+1)\pi}{2n}}{\cos \frac{\pi}{2n}}x + \frac{\sin \frac{(2j+1)\pi}{2n}}{\cos \frac{\pi}{2n}}y \leq 0\} \Big].$\\
Again, for each $ j \in \{1,2,\ldots, \frac{n-1}{2}\}, $ we have\\ 
$ v_{(j +\frac{n+1}{2})}^{\bot}= (-\cos\frac{j\pi}{n}, \sin\frac{j\pi}{n},1)^{\bot} = \pm \Big[ \{(x,y,z) \in \mathbb{X} : x \geq 0, y \geq 0, z \geq 0, $ \\ $ -\frac{\cos \frac{(2j+1)\pi}{2n}}{2\cos \frac{\pi}{2n}}x + \frac{\sin \frac{(2j+1)\pi}{2n}}{2\cos \frac{\pi}{2n}}y + \frac{z}{2} \geq 0, -\frac{\cos \frac{(2j-1)\pi}{2n}}{\cos \frac{\pi}{2n}}x + \frac{\sin \frac{(2j-1)\pi}{2n}}{\cos \frac{\pi}{2n}}y \leq 0\} \\ 
\cup \{(x,y,z) \in \mathbb{X} : x \leq 0, y \leq 0, z \geq 0, -\frac{\cos \frac{(2j-1)\pi}{2n}}{2\cos \frac{\pi}{2n}}x + \frac{\sin \frac{(2j-1)\pi}{2n}}{2\cos \frac{\pi}{2n}}y + \frac{z}{2} \geq 0, $ \\ 
$ -\frac{\cos \frac{(2j+1)\pi}{2n}}{\cos \frac{\pi}{2n}}x + \frac{\sin \frac{(2j+1)\pi}{2n}}{\cos \frac{\pi}{2n}}y \leq 0\}\\
\cup \{(x,y,z) \in \mathbb{X} : x \geq 0, y \leq 0, z \geq 0, -\frac{\cos \frac{(2j+1)\pi}{2n}}{2\cos \frac{\pi}{2n}}x + \frac{\sin \frac{(2j+1)\pi}{2n}}{2\cos \frac{\pi}{2n}}y + \frac{z}{2} \geq 0\} \\
\cup \{(x,y,z) \in \mathbb{X} : x \geq 0, y \leq 0, z \geq 0, -\frac{\cos \frac{(2j-1)\pi}{2n}}{2\cos \frac{\pi}{2n}}x + \frac{\sin \frac{(2j-1)\pi}{2n}}{2\cos \frac{\pi}{2n}}y + \frac{z}{2} \geq 0\} \Big].$\\
Now, $ v_{(n+1)}^{\bot} = (-1,0,1)^{\bot} = \pm \Big[ \{(x,y,z) \in \mathbb{X} : x \geq 0, y \geq 0, z \geq 0,$ \\ 
$ -\frac{x}{2}+ \frac{\tan(\frac{\pi}{2n})y}{2} + \frac{z}{2} \geq 0 \} \\
\cup \{(x,y,z) \in \mathbb{X} : x \leq 0, y \geq 0, z \geq 0, -x- \tan(\frac{\pi}{2n})y \leq 0 \} \\
\cup \{(x,y,z) \in \mathbb{X} : x \leq 0, y \leq 0, z \geq 0, -x+\tan(\frac{\pi}{2n})y \leq 0 \} \\ 
\cup \{(x,y,z) \in \mathbb{X} : x \geq 0, y \leq 0, z \geq 0, -\frac{x}{2} - \frac{\tan(\frac{\pi}{2n})y}{2} + \frac{z}{2} \geq 0\} \Big]. $\\
Again, for each $ j \in \{1,2,\ldots, \frac{n-1}{2}\}, $ we have\\ 
$ v_{(j+n+1)}^{\bot} = (-\cos\frac{j\pi}{n}, -\sin\frac{j\pi}{n},1)^{\bot} = \pm \Big[ \{(x,y,z) \in \mathbb{X} : x \geq 0, y \geq 0, z \geq 0, $ \\ 
$ -\frac{\cos \frac{(2j+1)\pi}{2n}}{2\cos \frac{\pi}{2n}}x - \frac{\sin \frac{(2j+1)\pi}{2n}}{2\cos \frac{\pi}{2n}}y + \frac{z}{2} \geq 0\} \\
\cup \{(x,y,z) \in \mathbb{X} : x \geq 0, y \geq 0, z \geq 0, -\frac{\cos \frac{(2j-1)\pi}{2n}}{2\cos \frac{\pi}{2n}}x - \frac{\sin \frac{(2j-1)\pi}{2n}}{2\cos \frac{\pi}{2n}}y + \frac{z}{2} \geq 0\} \\
\cup \{(x,y,z) \in \mathbb{X} : x \leq 0, y \geq 0, z \geq 0, -\frac{\cos \frac{(2j-1)\pi}{2n}}{2\cos \frac{\pi}{2n}}x - \frac{\sin \frac{(2j-1)\pi}{2n}}{2\cos \frac{\pi}{2n}}y + \frac{z}{2} \geq 0, $ \\
$ -\frac{\cos \frac{(2j+1)\pi}{2n}}{\cos \frac{\pi}{2n}}x - \frac{\sin \frac{(2j+1)\pi}{2n}}{\cos \frac{\pi}{2n}}y \leq 0\} \\ 
\cup \{(x,y,z) \in \mathbb{X} : x \geq 0, y \leq 0, z \geq 0, -\frac{\cos \frac{(2j+1)\pi}{2n}}{2\cos \frac{\pi}{2n}}x - \frac{\sin \frac{(2j+1)\pi}{2n}}{2\cos \frac{\pi}{2n}}y + \frac{z}{2} \geq 0, $ \\ 
$ -\frac{\cos \frac{(2j-1)\pi}{2n}}{\cos \frac{\pi}{2n}}x - \frac{\sin \frac{(2j-1)\pi}{2n}}{\cos \frac{\pi}{2n}}y \leq 0\} \Big].$\\
Also, for each $ j \in \{1,2,\ldots, \frac{n-1}{2}\}, $ we have\\
$ v_{(j + \frac{3n+1}{2})}^{\bot} = (\cos\frac{j\pi}{n}, -\sin\frac{j\pi}{n},1)^{\bot} = \pm \Big[ \{(x,y,z) \in \mathbb{X} : x \geq 0, y \geq 0, z \geq 0, $ \\
$ \frac{\cos \frac{(2j-1)\pi}{2n}}{2\cos \frac{\pi}{2n}}x - \frac{\sin \frac{(2j-1)\pi}{2n}}{2\cos \frac{\pi}{2n}}y + \frac{z}{2} \geq 0, \frac{\cos \frac{(2j+1)\pi}{2n}}{\cos \frac{\pi}{2n}}x - \frac{\sin \frac{(2j+1)\pi}{2n}}{\cos \frac{\pi}{2n}}y \leq 0\} \\
\cup \{(x,y,z) \in \mathbb{X} : x \leq 0, y \geq 0, z \geq 0, \frac{\cos \frac{(2j+1)\pi}{2n}}{2\cos \frac{\pi}{2n}}x - \frac{\sin \frac{(2j+1)\pi}{2n}}{2\cos \frac{\pi}{2n}}y + \frac{z}{2} \geq 0\} \\
\cup \{(x,y,z) \in \mathbb{X} : x \leq 0, y \geq 0, z \geq 0, \frac{\cos \frac{(2j-1)\pi}{2n}}{2\cos \frac{\pi}{2n}}x - \frac{\sin \frac{(2j-1)\pi}{2n}}{2\cos \frac{\pi}{2n}}y + \frac{z}{2} \geq 0\} \\ 
\cup \{(x,y,z) \in \mathbb{X} : x \leq 0, y \leq 0, z \geq 0, \frac{\cos \frac{(2j+1)\pi}{2n}}{2\cos \frac{\pi}{2n}}x - \frac{\sin \frac{(2j+1)\pi}{2n}}{2\cos \frac{\pi}{2n}}y + \frac{z}{2} \geq 0, $ \\
$ \frac{\cos \frac{(2j-1)\pi}{2n}}{\cos \frac{\pi}{2n}}x - \frac{\sin \frac{(2j-1)\pi}{2n}}{\cos \frac{\pi}{2n}}y \leq 0\} \Big].$\\
Now, $ w_{1}^{\bot} = (0,0,2)^{\bot}= \pm \Big[A \cup B \cup C \cup D\Big], $ where $ A = \bigcup \limits_{j=0}^{\frac{n-1}{2}} A_j, $ 
$ A_j =  \{(x,y,z) \in \mathbb{X} : x \geq 0, y \geq 0, z \geq 0, -\frac{\cos \frac{(2j+1)\pi}{2n}}{2\cos \frac{\pi}{2n}}x - \frac{\sin \frac{(2j+1)\pi}{2n}}{2\cos \frac{\pi}{2n}}y + \frac{z}{2} \leq 0 \}, $
$ B = \bigcup \limits_{j=0}^{\frac{n-1}{2}} B_j, $ 
$ B_j = \{(x,y,z) \in \mathbb{X} : x \leq 0, y \geq 0, z \geq 0, \frac{\cos \frac{(2j+1)\pi}{2n}}{2\cos \frac{\pi}{2n}}x - \frac{\sin \frac{(2j+1)\pi}{2n}}{2\cos \frac{\pi}{2n}}y + \frac{z}{2} \leq 0 \}, $
$ C = \bigcup \limits_{j=0}^{\frac{n-1}{2}} C_j, $ 
$ C_j = \{(x,y,z) \in \mathbb{X} : x \leq 0, y \leq 0, z \geq 0, \frac{\cos \frac{(2j+1)\pi}{2n}}{2\cos \frac{\pi}{2n}}x + \frac{\sin \frac{(2j+1)\pi}{2n}}{2\cos \frac{\pi}{2n}}y + \frac{z}{2} \leq 0 \}, $
$ D = \bigcup \limits_{j=0}^{\frac{n-1}{2}} D_j, $
$ D_j =  \{(x,y,z) \in \mathbb{X} : x \geq 0, y \leq 0, z \geq 0, -\frac{\cos \frac{(2j+1)\pi}{2n}}{2\cos \frac{\pi}{2n}}x + \frac{\sin \frac{(2j+1)\pi}{2n}}{2\cos \frac{\pi}{2n}}y + \frac{z}{2} \leq 0 \}. $
From the above expressions of the Birkhoff-James orthogonality sets of extreme points of $ B_{\mathbb{X}}, $ it follows that, for any two extreme points $ u, v \in S_{\mathbb{X}}, $ $ u^{\bot} \cup v^{\bot} \subsetneqq \mathbb{X}. $ Therefore $ \mathbb{X} $ has Property $ P_2. $\\
Again, by using the same expressions, we can show that $ (1,0,1)^{\bot} \cup (-1,0,1)^{\bot} \cup (0,0,2)^{\bot} = \mathbb{X}. $ Therefore, $ \mathbb{X} $ does not have Property $ P_3. $ This completes the proof of the theorem.
\end{proof}

Finally, we give an example of a linear operator $ T $ between a three-dimensional polyhedral Banach space and a three-dimensional polyhedral Banach space having Property $P_2, $ such that $ T $ does not satisfy the B\v{S} Property. 
\begin{example} \label{example-4}
	Let $ \mathbb{X} = \ell_{\infty}^{3} $ and let $ \mathbb{Y} $ be a three-dimensional polyhedral Banach space such that $ B_{\mathbb{Y}} $ is a polyhedron with vertices $ (1,0,\pm 1), (\frac{1}{\sqrt{2}}, \frac{1}{\sqrt{2}}, \pm 1),$ \\
	$ (0,1,\pm 1), (\frac{-1}{\sqrt{2}}, \frac{1}{\sqrt{2}}, \pm 1), (-1,0,\pm 1), (\frac{-1}{\sqrt{2}}, \frac{-1}{\sqrt{2}}, \pm 1), (0,-1,\pm 1), (\frac{1}{\sqrt{2}}, \frac{-1}{\sqrt{2}}, \pm 1),$ \\ $ (0,0,\pm 2) . $ Consider a bounded linear operator $ T : \mathbb{X} \to \mathbb{Y}, $ defined by 
	$$ T (x,y,z)= \Big(\frac{x+y}{2}, \frac{y-x}{2}, y \Big) . $$
	Then it is easy to check that $ \|T\|=1, $ $M_T= \{ \pm(1,1,z), \pm(-1,1,z) : z \in [-1,1] \} $ and $ T(M_T) = \{ \pm(1,0,1), \pm(0, 1, 1) \}. $ Consider $ x_1= (1,1,1), x_2= (-1,1,1) $ and $ x_3=(-1,-1,1). $ Clearly $ \{x_1,x_2,x_3\} $ forms a basis of $ \mathbb{X}. $ If we choose $ \alpha= -10 $ and $ \beta= \frac{-3}{2}, $ then condition (c) of Corollary \ref{cor:BS-n-dim} is satisfied. From Theorem \ref{th:prism-pyramid-general}, we know that $ \mathbb{Y} $ has Property $ P_2. $ Therefore, by using Corollary \ref{cor:BS-n-dim}, we conclude that $ T $ does not satisfy the B\v{S} Property.
\end{example}

In view of the methods employed to study the B\v{S} Property of linear operators and the results obtained in the present article, it is perhaps appropriate to end it with the following remark:

\begin{remark}
We have illustrated the important role played by Property $ P_n $ in determining the B\v{S} Property of linear operators. Indeed, using this concept, we have extended the previously obtained results in \cite{SPH}. It is worth mentioning in this connection that Example \ref{example-2}, Example \ref{example-3} and Example \ref{example-4}  provided in this article are beyond the scope of the Proposition $ 2.1 $ of \cite{SPH}. We note that Property $ P_n $ is essentially a structural concept, associated especially with polyhedral Banach spaces. Therefore, it might be interesting to further study various polyhedral Banach spaces in light of the newly introduced concept of Property $ P_n. $
\end{remark}

\bibliographystyle{amsplain}

\begin{thebibliography}{99}


\bibitem{BS} R. Bhatia and P. $\check{S}$emrl, \textit{Orthogonality of matrices and some distance problems},  Linear Algebra Appl., \textbf{287} (1999), no. 1--3, 77--85.
  
	
\bibitem{J} R. C. James, \textit{Orthogonality and linear functionals in normed linear spaces}, Trans. Amer. Math. Soc., \textbf{61} (1947), 265--292.

\bibitem{LS} C. K. Li and H. Schneider, \textit{Orthogonality of matrices}, Linear Algebra Appl., \textbf{347} (2002), 115--122.

\bibitem{P} K. Paul,
\textit{ Translatable radii of an operator in the direction of another operator},
Sci. Math.,
\textbf{2} (1999), no. 1, 119--122.

\bibitem{PHD} K. Paul, Sk. M. Hossein and K. C. Das,  \textit{Orthogonality on B(H,H) and Minimal-norm Operator},
  J. Anal. Appl., \textbf{6} (2008), no. 3, 169--178.
	
	\bibitem{Sa} D. Sain, \textit{ Birkhoff-James orthogonality of linear operators on finite dimensional Banach spaces}, J. Math. Anal. Appl., \textbf{447} (2017), 860--866.

\bibitem{SP} D. Sain and  K. Paul, \textit{ Operator norm attainment and inner product spaces}, Linear Algebra Appl., \textbf{439} (2013), no. 8, 2448--2452.

\bibitem{SPBB} D. Sain, K. Paul, P. Bhunia and S. Bag, \textit{ On the numerical index of polyhedral Banach spaces}, Linear Algebra Appl., \textbf{577} (2019), 121--133.
 
\bibitem{SPH} D. Sain, K. Paul and S. Hait, \textit{ Operator norm attainment and Birkhoff-James orthogonality}, Linear Algebra Appl., \textbf{476} (2015), 85--97.












\end{thebibliography}

\end{document}